\documentclass[11pt]{amsart}

\usepackage{amsfonts}
\usepackage{amsmath,amssymb,amsthm}
\usepackage{enumerate,graphicx}
\usepackage[square,numbers]{natbib}
\usepackage{MnSymbol}
\usepackage{mathrsfs,bbm} 

\newtheorem{theorem}{Theorem}[section]

\newtheorem{proposition}[theorem]{Proposition}
\newtheorem{corollary}[theorem]{Corollary}

\theoremstyle{definition}

\newtheorem{definition}[theorem]{Definition}
\newtheorem{example}[theorem]{Example}

\newtheorem{problem}{Problem}
\newtheorem{remark}[theorem]{Remark}

\begin{document}

\title{Mixing properties of tree-shifts}


\author[Jung-Chao Ban]{Jung-Chao Ban}
\address[Jung-Chao Ban]{Department of Applied Mathematics, National Dong Hwa University, Hualien 970003, Taiwan, ROC.}
\email{jcban@mail.ndhu.edu.tw}

\author[Chih-Hung Chang]{Chih-Hung Chang*}
\address[Chih-Hung Chang]{Department of Applied Mathematics, National University of Kaohsiung, Kaohsiung 81148, Taiwan, ROC.}
\email{chchang@nuk.edu.tw}

\date{October 13, 2016}

\thanks{*Correspondence to whom should be addressed.}
\thanks{This work is partially supported by the Ministry of Science and Technology, ROC (Contract No MOST 105-2115-M-390 -001 -MY2).}

\baselineskip=1.2\baselineskip

\begin{abstract}
The uncertainty of multidimensional shift spaces draws attracted attention of researchers. For example, the emptiness problem is undecidable; there exist aperiodic shifts of finite type; there is a nonempty shift of finite type exhibiting nonextensible local patterns. This paper investigates symbolic dynamics on Cayley trees and gives affirmative answers to the above questions in tree-shifts. Beyond that, with introducing block gluing tree-shift, a sufficient condition of exhibiting positive topological entropy is revealed.
\end{abstract}

\maketitle

\section{Introduction} \label{sec:intro}

There have been many researches about chaotic systems, such as the strange attractor in the Lorenz system, period doubling in quadratic maps, and Julia sets in complex-valued functions, over the past few decades. For most systems, the theoretical analysis of the chaotic behavior is difficult; one of the most frequently used techniques is transferring the original system to a conjugate or semiconjugate symbolic dynamical system and then investigating the chaotic behavior in symbolic dynamics (see \cite{Dev-1987} and the references therein).

A shift space is a set consisting of right-infinite or bi-infinite words which avoid those finite words in a so-called \emph{forbidden set} $\mathcal{F}$. Such a shift space is denoted by $\mathsf{X}_{\mathcal{F}}$. A shift space $\mathsf{X}_{\mathcal{F}}$ is called a shift of finite type (SFT) whenever $\mathcal{F}$ is finite. Shifts of finite type are fundamental and play an important role in symbolic dynamical systems. Roughly speaking, a shift of finite type is a set of right-infinite or bi-infinite paths in a finite directed graph. Moreover, investigating the graph representation of a shift of finite type reveals some important properties such as irreducibility, mixing, and spatial chaos (see \cite{Kit-1998, LM-1995}).

Suppose that a forbidden set $\mathcal{F}$ is given. It comes immediately to mind plenty questions. The following are some frequently asked ones.
\begin{enumerate}[(a)]
\item Does there exist an algorithm to determine if $\mathsf{X}_{\mathcal{F}}$ is empty?
\item Does every nonempty shift space contain a periodic point? Under what condition are periodic points dense?
\item Is every local pattern extensible?
\end{enumerate}
The answers to the above questions are affirmative, whenever we focus on SFTs. A SFT is nonempty if its associative graph representation is essential. Every nonempty SFT contains periodic points and an irreducible SFT has dense periodic points. Furthermore, every local pattern can extend globally. Nevertheless, these questions receive opposite results when investigating multidimensional shift spaces.

For two-dimensional SFTs, the emptiness problem is undecidable; that is, there is no algorithm for determining if $\mathsf{X}_{\mathcal{F}}$ is empty for a given forbidden set. There is an aperiodic SFT which has positive topological entropy, and there is a nonempty SFT which exhibits nonextensible local patterns. These results reveal the uncertainty of multidimensional shift spaces and attracted attention of researchers (cf.~\cite{Berger-MAMS1966, CulikII-DM1996, GK-SMJ1972, Kari-DM1996, Robinson-IM1971, Schmidt-2001}). Sharma and Kumar \cite{SK-2016} have demonstrated the necessary and sufficient condition for determining if a multidimensional SFT is empty; furthermore, they address a sufficient condition for multidimensional SFTs exhibiting periodic points.

While every one-dimensional mixing SFT or sofic shift of positive topological entropy contains a tremendous collection of pairwise disjoint subsystems, Desai \cite{Des-IM2006} has shown that a general multidimensional sofic shift $X$ of positive entropy $h(X)$ still contains subshifts achieving all entropies in the interval $[0, h(X)]$. In \cite{BPS-TAMS2010}, Boyle \emph{et al.}~demonstrate that these subshifts may be poorly separated; the same phenomenon is also observed in multidimensional SFTs. They also introduce a mixing condition known as \emph{block gluing} and distinguish several different mixing conditions such as \emph{uniform filling} and \emph{strongly irreducible}. In addition, they show that every two-dimensional block gluing SFT has dense periodic points; however, the denseness of periodic points in general multidimensional SFTs remains to be open. For more details about the recent works in multidimensional shift spaces, the reader is referred to \cite{BHL+-2015, BFM-IJAC2005, BPS-TAMS2010, Briceno-ETDS2016, JM-ETDS2005, Lightwood-ETDS2003, MP-1979, MP-PLMS1981, MP-2016, PS-TAMS2015, Schraudner-DCDS2010, Ward-IMN1994} and the references therein.

Topological entropy, which represents the spatial growth rate of number of points (patterns), is one of the most important quantities which reflects the complexity of a dynamical system. The collection of the topological entropies of one-dimensional SFTs is known as the set of Perron numbers (together with a logarithmic function). Hochman and Meyerovitch \cite{HM-AoM2010} have indicated that the topological entropy of a multidimensional SFT is right recursively enumerable; roughly speaking, it is the infimum of a monotonic recursive sequence of rational numbers. Hochman extends the result to the multidimensional effective dynamical systems \cite{Hochm-IM2009}. Pavlov and Schraudner \cite{PS-JdM2015} show that, for every $d \geq 3$ and every $\mathbb{Z}^d$ full shift, there is a block gluing $\mathbb{Z}^d$ SFT which shares identical topological entropy. Although there exist multidimensional SFTs with arbitrary high topological entropy, which do not factor onto any full shift, a block gluing shift space factors onto every shift space of strictly lower entropy under some conditions; furthermore, every nontrivial block gluing shift space is of positive entropy (cf.~\cite{BPS-TAMS2010, Briceno-ETDS2016, MP-2016}).

One of the reasons which causes these differences between one- and multidimensional shift spaces is the spatial structure; one-dimensional lattice $\mathbb{Z}$ is a free group with one generator while multidimensional lattice $\mathbb{Z}^d$, $d \geq 2$, is an abelian group with $d$ generators and has loops itself. This motivates the investigation of symbolic dynamics on Cayley trees. In \cite{AB-TCS2012, AB-TCS2013}, the authors introduce the notion of shifts defined on Cayley trees, which are called tree-shifts. Tree-shifts have a natural structure of one-sided symbolic dynamical systems equipped with multiple shift maps. The $i$th shift map applies to a tree that gives the subtree rooted at the $i$th children of the tree. Sets of finite patterns of tree-shifts of finite type are strictly testable tree languages. Such testable tree languages are also called $k$-testable tree languages. Probabilistic $k$-testable models are used for pattern classification and stochastic learning (cf.~\cite{VCC-PAMIIEEET2005}). It is remarkable that M\"{u}ller and Spandl demonstrate the existence of embedding maps from a topological dynamical system on metric Cantor space to a cellular automaton defined on Cayley graph, which preserves topological entropy \cite{MS-ETDS2009}.

The conjugacy of multidimensional shifts of finite type (also known as textile systems or tiling systems) is undecidable (see \cite{CJJ+-ETDS2003, JM-PAMS1999, LS-2002} and the references therein). Namely, there is no algorithm for determining whether two tiling systems are conjugate. Nevertheless, Williams indicates that the conjugacy of one-sided shifts of finite type is decidable \cite{Wil-AoM1973}. Aubrun and B\'{e}al extend Williams' result to tree-shifts; more precisely, they show that the conjugacy of irreducible tree-shifts of finite type are decidable \cite{AB-TCS2012}. Furthermore, Aubrun and B\'{e}al accomplish other celebrated results in tree-shifts, such as realizing tree-shifts of finite type and sofic tree-shifts via tree automata, developing an algorithm for determining whether a sofic tree-shift is a tree-shift of finite type, and the existence of irreducible sofic tree-shifts that are not the factors of tree-shifts of finite type (TSFTs). The reader is referred to \cite{AB-TCS2012, AB-TCS2013} for more details.

In \cite{BC-2015}, we show that every irreducible tree-shift of finite type and every mixing tree-shift has dense periodic points, which is an extension of the results in one-sided and multidimensional shift spaces. Such a phenomenon also reveals that the tree-shifts constitute an intermediate class in classical symbolic dynamics. In addition, we show that every TSFT is conjugate to a vertex tree-shift (defined later), which is a TSFT represented by finitely many $0$-$1$ matrices. After illustrating that computing the topological entropy of a TSFT is equivalent to the investigation of a system of nonlinear recurrence equations \cite{BC-2015a}, Akiyama \emph{et al.} indicate that the collection of topological entropies of TSFTs, like one-dimensional SFTs, is the set of Perron numbers \cite{ABC-2016}.

In this paper, we investigate the following fundamental problems in tree-shifts. Suppose that a forbidden set $\mathcal{F}$ is given.
\begin{enumerate}[(a)]
\item Does there exist an algorithm to determine if the tree-shift $\mathsf{X}_{\mathcal{F}}$ is empty?
\item Does every nonempty shift space contain a periodic point?
\item Is every local pattern extensible?
\item Under what condition is a tree-shift of positive topological entropy?
\end{enumerate}
To illustrate the above questions, we develop several mixing conditions for tree-shifts, based on the conditions discussed in \cite{BPS-TAMS2010}, such as topological mixing (TM), blocking gluing (BG), and strongly irreducible (SI) (defined later).  The main results of this paper are as follows.
\begin{enumerate}[(a)]
\item Every essential vertex tree-shift is nonempty, and vice versa. Roughly speaking, an essential vertex tree-shift is a TSFT represented by essential matrices.
\item Every essential vertex tree-shift contains periodic points.
\item Every local pattern of an essential vertex tree-shift is extensible.
\item Every nontrivial block gluing TSFT is of positive topological entropy.
\end{enumerate}
In addition, two extra mixing conditions called \emph{uniformly block gluing} (UBG) and \emph{uniformly strongly irreducible} (USI) are introduced. The relations between these mixing conditions are revealed as follows.
$$
\begin{array}{ccccc}
USI & \Rightarrow & SI &  &  \\
\Downarrow & \rotatebox{45}{$\nLeftrightarrow$} & \Downarrow &  &  \\
UBG & \Rightarrow & BG & \Rightarrow & TM
\end{array}
$$
Nevertheless, the diagram reduces to
$$
UBG (=USI) \Rightarrow BG (=SI) \Rightarrow TM
$$
whenever we restrict the discussion to TSFTs. More specifically, a TSFT is strongly irreducible (respectively uniformly strongly irreducible) if and only if it is block gluing  (respectively uniformly block gluing). This is another difference between tree-shifts and multidimensional shift spaces.

The materials of this work are organized as follows. The upcoming section addresses the background of the tree-shifts; an equivalent property of tree-shift is elucidated therein. The main results are discussed in Section \ref{sec:main-result} while Section \ref{sec:conclusion} gives the conclusion and summary of the present work. Some further questions of interest are also indicated.

\section{Definitions and Previous Results} \label{sec:preliminary}

This section recalls some basic definitions of symbolic dynamics on Cayley trees. More explicitly, the nodes of infinite trees considered in this paper have a fixed number of children and are labeled in a finite alphabet. To clarify the discussion, we focus on binary trees, but all results extend to the case of trees with $d$ children for a fixed positive integer $d$. In other words, the class of classical one-sided shift spaces is a special case in the present study.

\subsection{Basic definitions}
Let $\Sigma = \{0, 1\}$ and let $\Sigma^* = \bigcup_{n \geq 0} \Sigma^n$ be the set of words over $\Sigma$, where $\Sigma^n = \{x_1 x_2 \cdots x_n: x_i \in \Sigma \text{ for } 1 \leq i \leq n\}$ is the set of words of length $n$ for $n \in \mathbb{N}$ and $\Sigma^0 = \{\epsilon\}$ consists of the empty word $\epsilon$. An \emph{infinite tree} $t$ over a finite alphabet $\mathcal{A}$ is a function from $\Sigma^*$ to $\mathcal{A}$; a \emph{node} of an infinite tree is a word of $\Sigma^*$, and the empty word relates to the \emph{root} of the tree. Suppose $x$ is a node of a tree. Each node $xi$, $i \in \Sigma$, is known as a \emph{child} of $x$ while $x$ is the \emph{parent} of $xi$. A sequence of words $(w_k)_{1 \leq k \leq n}$ is called a \emph{path} if, for all $k \leq n-1$, $w_{k+1} = w_k i_k$ for some $i_k \in \Sigma$ and $w_1 \in \Sigma^*$. For the rest of this investigation, a tree is referred to as an infinite tree unless otherwise stated.

Let $t$ be a tree and let $x$ be a node, we refer $t_x$ to $t(x)$ for simplicity. A \emph{subtree} of a tree $t$ rooted at a node $x$ is the tree $t'$ satisfying $t'_y = t_{xy}$ for all $y \in \Sigma^*$ such that $xy$ is a node of $t$, where $xy = x_1 \cdots x_m y_1 \cdots y_n$ means the concatenation of $x = x_1 \cdots x_m$ and $y_1 \cdots y_n$. Given two words $x = x_1 x_2 \ldots x_i$ and $y = y_1 y_2 \ldots y_j$, we say that $x$ is a \emph{prefix} of $y$ if and only if $i \leq j$ and $x_k = y_k$ for $1 \leq k \leq i$. A subset of words $L \subset \Sigma^*$ is called \emph{prefix-closed} if each prefix of $L$ belongs to $L$. A function $u$ defined on a finite prefix-closed subset $L$ with codomain $\mathcal{A}$ is called a \emph{pattern}, and $L$ is called the \emph{support} of the pattern; a pattern is called a \emph{block of height $n$} (or an \emph{$n$-block}) if its support $L = x \Sigma_{n-1}$ for some $x \in \Sigma^*$, where $\Sigma_i = \bigcup\limits_{0 \leq k \leq i} \Sigma^k$.

Suppose that $u$ is a pattern and $t$ is a tree. Let $S(u)$ denote the support of $u$. We say that $u$ is accepted by $t$ if there exists $x \in \Sigma^*$ such that $u_y = t_{xy}$ for every node $y \in S(u)$. In this case, we say that $u$ is a pattern of $t$ rooted at the node $x$. A tree $t$ is said to \emph{avoid} $u$ if $u$ is not accepted by $t$; otherwise, $u$ is called an \emph{allowed pattern} of $t$.

We denote by $\mathcal{T}$ (or $\mathcal{A}^{\Sigma^*}$) the set of all infinite trees on $\mathcal{A}$. For $i \in \Sigma$, the shift transformations $\sigma_i$ from $\mathcal{T}$ to itself are defined as follows. For every tree $t \in \mathcal{T}$, $\sigma_i(t)$ is the tree rooted at the $i$th child of $t$, that is, $(\sigma_i(t))_x = t_{ix}$ for all $x \in \Sigma^*$. For the purpose of simplification of the notation, we omit the parentheses and denote $\sigma_i(t)$ by $\sigma_i t$. The set $\mathcal{T}$ equipped with the shift transformations $\sigma_i$ is called the \emph{full tree-shift} of infinite trees over $\mathcal{A}$. Suppose $w = w_1 \cdots w_n \in \Sigma^*$. Define $\sigma_w = \sigma_{w_n} \circ \sigma_{w_{n-1}} \circ \cdots \circ \sigma_{w_1}$; it follows immediately that $(\sigma_w t)_x = t_{wx}$ for all $x \in \Sigma^*$.

Given a collection of patterns $\mathcal{F}$, let $\mathsf{X}_{\mathcal{F}}$ denote the set of trees avoiding any element of $\mathcal{F}$. A subset $X \subseteq \mathcal{T}$ is called a \emph{tree-shift} if $X = \mathsf{X}_{\mathcal{F}}$ for some $\mathcal{F}$. We say that $\mathcal{F}$ is \emph{a set of forbidden patterns} (or \emph{a forbidden set}) of $X$. A straightforward examination suggests that a tree-shift is closed and $\sigma_w$-invariant for all $w \in \Sigma^*$. For each given set of forbidden patterns, the emptiness  problem immediately follows.

\begin{problem}[Emptiness Problem] \label{prob:empty-problem}
Given a set of finite patterns $\mathcal{F} \subset \mathcal{A}^{\Sigma^*}$, determine whether $\mathsf{X}_{\mathcal{F}} = \varnothing$.
\end{problem}

Denote the set of all blocks of height $n$ of $X$ by $B_n(X)$, and denote the set of all blocks of $X$ by $B(X)$. Suppose $u \in B_n(X)$ for some $n \geq 2$. Let $\sigma_i u$ be the block of height $n-1$ such that $(\sigma_i u)_x = u_{ix}$ for $x \in \Sigma_{n-2}$. The block $u$ is written as $u = (u_{\epsilon}, \sigma_0 u, \sigma_1 u)$.

A set of patterns $L$ is called \emph{factorial} if $u \in L$ and $v$ is a sub-pattern of $u$ implies $v \in L$. We say that $v$ is a sub-pattern of $u$ if $v$ is a subtree of $u$ rooted at some node $x$ of $u$. The set $L$ is called \emph{extensible} if for any pattern $u \in L$ with support $S(u)$, there exists a pattern $v \in L$ with support $S(v)$ such that $S(u) \subset S(v)$, $v$ coincides with $u$ on $S(u)$, and for any $x \in S(u)$, we have $xi \in S(v)$ for all $i \in \Sigma$. Notably, given a set of patterns $\mathcal{F}$, a pattern $u$ is extensible if and only if $u$ is an allowed pattern of some tree $t \in \mathsf{X}_{\mathcal{F}}$. Suppose that $\mathsf{X}_{\mathcal{F}} \neq \varnothing$. It is of interest that if every local pattern $u \notin \mathcal{F}$ extends to a global pattern, namely, a tree $t \in \mathsf{X}_{\mathcal{F}}$.

\begin{problem}[Extensibility Problem] \label{prob:extend-problem}
Let $\mathcal{F} \subset \mathcal{A}^{\Sigma^*}$ be a set of finite patterns and let $u \notin \mathcal{F}$ be a finite pattern. Does there exist $t \in \mathsf{X}_{\mathcal{F}}$ such that $u$ is accepted by $t$?
\end{problem}

For any two trees $t$ and $t'$, define
\begin{equation}
\mathrm{d}(t, t') = \left\{
             \begin{array}{ll}
               2^{-n}, & n = \min\{|x|: t_x \neq t'_x\} < \infty \hbox{;} \\
               0, & \hbox{otherwise;}
             \end{array}
           \right.
\end{equation}
herein, $|x|$ means the length of $x$. Then $\mathrm{d}$ is a metric on $\mathcal{T}$. Suppose $L$ is a factorial and extensible set of patterns. Let $\mathcal{X}(L)$ be the collection of trees whose patterns belong to $L$. Then $\mathcal{X}(L)$ is a tree-shift and $B(\mathcal{X}(L)) = L$. Conversely, if $X$ is a tree-shift, then $X = \mathcal{X}(B(X))$. This result is similar to the one known for the classical shift spaces. The reader is referred to \cite{AB-TCS2013, LM-1995} for more details.

\begin{theorem}
Let $X \subseteq \mathcal{A}^{\Sigma^*}$ be a collection of infinite trees. The following are equivalent.
\begin{enumerate}[(a)]
\item $X$ is closed and shift invariant; that is, $\sigma_w t \in X$ for all $w \in \Sigma^*$ and $t \in X$.
\item There exists a set of finite patterns $\mathcal{F}$ such that $X = \mathsf{X}_{\mathcal{F}}$.
\end{enumerate}
\end{theorem}
\begin{proof}
Suppose that $X$ is closed and shift invariant. Set $L = \bigcup\limits_{n \geq 1} B_n(X)$; the shift-invariance of $X$ infers that $L$ is factorial. The definition of $L$ indicates that $L$ is extensible and $X = \mathcal{X}(L)$. Therefore, $X$ is a tree-shift and $X = \mathsf{X}_{\mathcal{F}}$ for some $\mathcal{F}$.

Conversely, let $\mathcal{F}$ be a set of finite patterns and let $X = \mathsf{X}_{\mathcal{F}}$. It remains to show that $X$ is closed. Suppose that there exists $\{t_n\}_{n \geq 1} \subseteq X$ such that $t_n \to t$ as $n \to \infty$ and $t \notin X$. Then there is an allowed pattern $u$ of $t$ such that $u \in \mathcal{F}$. Without loss of generality, we may assume that $S(u) = \Sigma_k$ for some $k \in \mathbb{N}$. Since $\lim\limits_{n \to \infty} t_n = t$, there exists $m \in \mathbb{N}$ such that $t_m |_{\Sigma_k} = t|_{\Sigma_k}$, which leads to a contradiction. This completes the proof.
\end{proof}

Let $\mathcal{T}$ and $\mathcal{T}'$ be the full tree-shifts over finite alphabets $\mathcal{A}$ and $\mathcal{A}'$, respectively, and let $X$ be a tree-subshift of $\mathcal{T}$. (That is, $X$ is itself a tree-shift and $X \subseteq \mathcal{T}$.) A function $\phi: X \to \mathcal{T}'$ is called a \emph{sliding block code} if there exists a positive integer $m$ and a map $\Phi: B_m(X) \to \mathcal{A}'$ such that $\phi(t)_x = \Phi(u)$, the image of $m$-block of $t$ rooted at $x$ with respect to $\Phi$, for all $x \in \Sigma^*$. In this case, we denote $\phi = \Phi_{\infty}$. The local map $\Phi$ herein is called an \emph{$m$-block map}, and a \emph{block map} is a map which is an $m$-block map for some positive integer $m$.

In the theory of symbolic dynamics, the Curtis-Lyndon-Hedlund Theorem (see \cite{Hed-MST1969}) indicates that, for two shift spaces $X$ and $Y$, a map $\phi: X \to Y$ is a sliding block code if and only if $\phi$ is continuous and $\phi \circ \sigma_x = \sigma_Y \circ \phi$. A similar discussion extends to tree-shifts; in other words, $\phi$ is a sliding block code (between tree-shifts) if and only if $\phi$ is continuous and commutes with all tree-shift maps $\sigma_i$ for $i \in \Sigma$.

If a sliding block code $\phi: X \to Y$, herein $X$ and $Y$ are tree-shifts, is onto, then $\phi$ is called a \emph{factor code} from $X$ to $Y$. A tree-shift $Y$ is a \emph{factor} of $X$ if there is a factor code from $X$ onto $Y$. If $\phi$ is one-to-one, then $\phi$ is called an \emph{embedding} of $X$ into $Y$.

A sliding block code $\psi: Y \to X$ is called an \emph{inverse} of $\phi$ if $\psi(\phi(x)) = x$ for all $x \in X$ and $\phi(\psi(y)) = y$ for all $y \in Y$. In this case, we say that $\phi$ is \emph{invertible} and write $\psi = \phi^{-1}$.

\begin{definition}
A sliding block code $\phi: X \to Y$ is a \emph{conjugacy from $X$ to $Y$} if it is invertible. Two tree-shifts $X$ and $Y$ are called \emph{conjugate}, denoted by $X \cong Y$, if there is a conjugacy from $X$ to $Y$.
\end{definition}

A tree-shift $X = \mathsf{X}_{\mathcal{F}}$ is called a \emph{tree-shift of finite type} (TSFT) if the forbidden set $\mathcal{F}$ is finite; we say that $X$ is a \emph{sofic tree-shift} if $X$ is a factor of some TSFT. Given two $0$-$1$ matrices $A_0, A_1$ which are indexed by $\mathcal{A}$, the \emph{vertex tree-shift} $\mathsf{X}_{A_0, A_1}$ (corresponding to $A_0$ and $A_1$) is defined as
\begin{equation}
\mathsf{X}_{A_0, A_1} = \{t \in \mathcal{A}^{\Sigma^*}: A_0 (t_x, t_{x0}) = 1 \text{ and } A_1 (t_x, t_{x1}) = 1 \text{ for all } x \in \Sigma^*\}.
\end{equation}

It follows immediately from the definition that each vertex tree-shift is a TSFT. Our previous work demonstrates that every TSFT can be treated as a vertex tree-shift after recoding, which extends a classical result in symbolic dynamical systems.

\begin{proposition}[See \cite{BC-2015}]\label{prop:TSFT-is-vertex-shift-is-Markov}
Every tree-shift of finite type is conjugate to a vertex tree-shift.
\end{proposition}

Proposition \ref{prop:TSFT-is-vertex-shift-is-Markov} asserts that the discussion of tree-shifts of finite type are equivalent to investigating vertex tree-shifts. For the rest of this paper, a tree-shift of finite type is referred to as a vertex tree-shift unless otherwise stated.

\subsection{Existence and denseness of periodic points}

One of the most interested problems in the investigation of shift spaces is the existence of periodic points. While, for the one-dimensional case, every shift of finite type contains periodic points and the set of periodic points is dense in an irreducible shift of finite type, there is a two-dimensional shift of finite type which contains no periodic points. Before addressing the periodic points of the tree-shifts, we recall some definitions and results first.

\begin{definition}
Let $P \subset \Sigma^*$ be a subset of words. $P$ is called a \emph{prefix set} if no word in $P$ is a prefix of another one. The length of $P$, denoted by $|P|$, is the longest word in $P$. More specifically,
$$
|P| = \left\{
             \begin{array}{ll}
               \max\{|x|: x \in P\}, & \hbox{$P$ is a finite set;} \\
               \infty, & \hbox{otherwise.}
             \end{array}
           \right.
$$
A finite prefix set $P$ is called a \emph{complete prefix code} (CPC) if any $x \in \Sigma^*$, such that $|x| \geq |P|$, has a prefix in $P$.
\end{definition}

\begin{definition}[See \cite{AB-TCS2013}]\label{def:irreducible}
A tree-shift $X$ is \emph{irreducible} if for each pair of blocks $u, v$ with $u, v \in B_n(X)$, there is a tree $t \in X$ and a complete prefix code $P \subset \bigcup_{k \geq n} \Sigma^k$ such that $u$ is a subtree of $t$ rooted at $\epsilon$ and $v$ is a subtree of $t$ rooted at $x$ for all $x \in P$.
\end{definition}

An intuitional explanation of an irreducible tree-shift is that arbitrary two patterns can connect with one another, and a CPC is a bridge which connects the designated patterns. Aubrun and B\'{e}al \cite{AB-TCS2013} demonstrate that, for any two conjugate tree-shifts $X$ and $Y$, $X$ is irreducible if and only if $Y$ is irreducible. The definition of irreducible tree-shifts seems strong. However, it is seen that such a definition is natural in the way that it extends the theory of shift spaces to tree-shifts.

\begin{theorem}[See \cite{BC-2015}]\label{thm:equiv-def-irr}
Suppose $X$ is a tree-shift. The following are equivalent.
\begin{enumerate}[\bf (i)]
\item $X$ is irreducible.
\item For each pair of blocks $u \in B_n(X), v \in B_m(X)$, where $n, m \in \mathbb{N}$, there exists $\{P_w\}_{w \in \Sigma^{n-1}}$ with $P_w$ being a complete prefix code for any $w \in \Sigma^{n-1}$ and $t \in X$ such that
$$
t|_{S(u)} = u \quad \text{and} \quad t|_{w x S(v)} = v \text{ for all } w \in \Sigma^{n-1}, x \in P_w.
$$
\item For each pair of blocks $u \in B_n(X), v \in B_m(X)$, where $n, m \in \mathbb{N}$, there exists $\{P_k\}_{1 \leq k \leq l}$ for some $l$ with $P_k$ being a complete prefix code for $1 \leq k \leq l$ and $t \in X$ such that $t|_{S(u)} = u$ and, for each $w \in \Sigma^{n-1}$,
$$
t|_{w x S(v)} = v \text{ for all } x \in P_k \text{ for some } k.
$$
\end{enumerate}
\end{theorem}

Theorem \ref{thm:equiv-def-irr} reveals that Definition \ref{def:irreducible} is natural for tree-shifts and can extend to the definition of \emph{mixing tree-shifts} as one-dimensional symbolic dynamics do. More specifically, the main difference between irreducible and mixing tree-shifts is whether the CPC depends on the given patterns we want to connect together. For more details, the reader is referred to \cite{AB-TCS2013, BC-2015}.

Similar to the definition of irreducibility, a periodic point in a tree-shift is defined as follows.

\begin{definition}\label{def:periodic-tree}
Let $X$ be a tree-shift. An infinite tree $t \in X$ is \emph{periodic} if there is a complete prefix code $P$ such that $\sigma_x t=t$ for all $x \in P$, where $\sigma_x =\sigma_{x_k} \circ \sigma_{x_{k-1}} \circ \ldots \circ \sigma_{x_1}$ for $x=x_1 \ldots x_k$.
\end{definition}

\begin{theorem}[See \cite{BC-2015}]\label{thm:irr-TSFT-Dense-Periodic-Points}
Suppose that $X$ is a tree-shift. Then the periodic points of $X$ are dense in $X$ if $X$ is an irreducible tree-shift of finite type or $X$ is a mixing tree-shift.
\end{theorem}

Theorem \ref{thm:irr-TSFT-Dense-Periodic-Points} reveals that the periodic points are dense in a tree-shift provided it is either an irreducible TSFT or a mixing tree-shift. Such a result illustrates that tree-shifts are different from one-dimensional and multidimensional shift spaces due to the sufficient condition for the denseness of periodic points. Remarkably, every one-dimensional SFT contains periodic points while there exists an aperiodic two-dimensional SFT (cf.~\cite{BPS-TAMS2010}). It is of interest that if there exists an aperiodic TSFT.

\begin{problem}[Existence of Periodic Point]\label{prob:existence-periodic-point}
Does there exist an aperiodic tree-shift of finite type?
\end{problem}

One of the indicators which reflects the complexity of a dynamical system is topological entropy. The topological entropy of a multidimensional shift space is defined as the growth rate of the number of possible patterns with respect to the lattices. For a tree-shift $X$, we define its topological entropy as follows.

\begin{definition}
Suppose that $X$ is a tree-shift. The \emph{topological entropy} $h(X)$ of $X$ is defined as
\begin{equation}\label{eq:entropy-definition}
h(X) = \lim_{n \to \infty} \dfrac{\ln^2 |B_n(X)|}{n},
\end{equation}
where $\ln^2 = \ln \circ \ln$ and $|B_n(X)|$ is the cardinality of $B_n(X)$.
\end{definition}

The existence of the limit in \eqref{eq:entropy-definition} is demonstrated in \cite{ABC-2016, BC-2015a}. An immediate question then follows.

\begin{problem}[Postiivity of Topological Entropy]\label{prob:positive-entropy}
When is a tree-shift of positive topological entropy?
\end{problem}

\section{Main Results} \label{sec:main-result}

This section is devoted to the main results of this paper. After making inquiries about the emptiness and extensibility problems, the relations between different types of mixing properties are examined. The following elucidation focuses on binary tree-shifts (i.e., $|\Sigma| = 2$) to simplify the discussion, and can extend to general $d$-ary tree-shifts without difficulty.

We start with the emptiness problem for tree-shifts of finite type. Proposition \ref{prop:TSFT-is-vertex-shift-is-Markov} demonstrates that every tree-shift of finite type over $\mathcal{A}$ is a vertex tree-shift $\mathsf{X}_{A_0, A_1}$ for some binary matrices $A_0, A_1$ after recoding. Let $G_i = (V_i, E_i)$ be the graph representation of $A_i$ for $i = 0, 1$. Without loss of generality, we may assume that $A_0$ and $A_1$ are of the same dimension; it follows from the definition of vertex tree-shifts that $V_0 = V_1 = \mathcal{A}$.

A matrix is called \emph{essential} if it contains no zero rows. We say that $A_0$ and $A_1$ contain essential submatrices simultaneously if and only if there exist a permutation matrix $P$ and a natural number $k$ such that the first $k \times k$ blocks of $P^{-1} A_0 P$ and $P^{-1} A_1 P$ are both essential. More precisely, $B_0$ and $B_1$ are both essential matrices, where
$$
B_i (p, q) = (P^{-1} A_i P)(p, q) \quad \text{for} \quad 1 \leq p, q \leq k, i = 0, 1.
$$
Let $G = (V, E)$ be a graph. A vertex $v \in V$ is called \emph{stranded} if it is a sink; that is, a stranded vertex is a vertex which is not an initial state of any edge $e \in E$. A graph $G$ is called essential if it has no stranded vertices. A straightforward examination asserts that the graph representation of an essential matrix is essential, and vice versa. Given $V' \subset V$, a subgraph $G' = (V', E')$ of $G$ reduced by $V'$ is defined as $e \in E'$ if and only if $e \in E$ and the initial and terminal states of $e$ are both in $V'$. We denote such a subgraph by $G' = G|_{V'}$.

Theorem \ref{thm:emptiness} gives the emptiness problem (Problem \ref{prob:empty-problem}) an affirmative answer.

\begin{theorem}\label{thm:emptiness}
The following statements are equivalent.
\begin{enumerate}[\bf (i)]
\item $\mathsf{X}_{A_0, A_1} \neq \varnothing$.
\item $A_0$ and $A_1$ contain essential $k \times k$ submatrices simultaneously for some $k \in \mathbb{N}$.
\item There exists $V \subseteq V_0 = V_1$ such that $G_0|_V$ and $G_1|_V$ are essential graphs.
\end{enumerate}
\end{theorem}
\begin{proof}
Without loss of generality, we may assume that the adjacency matrices $A_0$ and $A_1$ are both indexed by the same order and $V_0 = V_1$ consists of $m$ vertices.

\noindent (i) $\Rightarrow$ (ii) Let $t \in \mathsf{X}_{A_0, A_1}$. Since $\mathcal{A}$ is finite, there exists a positive integer $j \leq m$ such that, for each $x \in \Sigma^j$, there exists $y \in \Sigma^i$ with $i < j$ such that $t_x = t_y$. Let $\mathcal{A}' = \{t_x: x \in \Sigma_j\}$, and let $A_{0; k}$ and $A_{1; k}$ be the matrices obtained from restricting $A_0$ and $A_1$ on $\mathcal{A}'$, respectively, where $k$ is the cardinality of $\mathcal{A}'$. It is seen that $A_{0; k}(t_x, t_{x0}) = A_0 (t_x, t_{x0}) = 1$ and $A_{1; k}(t_x, t_{x1}) = A_1 (t_x, t_{x1}) = 1$ for each $x \in \Sigma_{j-1}$; more explicitly, for each $a \in \mathcal{A}'$, there exist $a_0, a_1 \in \mathcal{A}'$ such that $A_{0; k}(a, a_0) = 1$ and $A_{1; k}(a, a_1) = 1$. This derives that $A_0$ and $A_1$ contain $k \times k$ submatrices which have no zero rows.

\noindent (ii) $\Rightarrow$ (i) Without loss of generality, we may assume that $k = m$. It follows from $A_0$ and $A_1$ containing no zero rows that, for each $a \in \mathcal{A}$, there exist $a_0, a_1 \in \mathcal{A}$ such that $A_0 (a, a_0) = 1$ and $A_1 (a, a_1) = 1$. Namely, for each $a \in \mathcal{A}$,  $u^{(2)} := (a, a_0, a_1) \in B_2(\mathsf{X}_{A_0, A_1})$ for some $a_0, a_1 \in \mathcal{A}$. Similarly, there exist $a_{00}, a_{01}, a_{10}, a_{11} \in \mathcal{A}$ such that $u^{(3)} \in B_3(\mathsf{X}_{A_0, A_1})$. Repeating the same process we can construct $u^{(n)} \in B_n(\mathsf{X}_{A_0, A_1})$ for any integer $n \geq 2$. Thus, $\mathsf{X}_{A_0, A_1} \neq \varnothing$.

Note that the demonstration of (ii) $\Leftrightarrow$ (iii) can be done straightforward. This completes the proof.
\end{proof}

Theorem \ref{thm:emptiness} indicates that the essential graphs are fundamental for the non-emptiness of tree-shifts of finite type. A vertex tree-shift $\mathsf{X}_{A_0, A_1}$ is called \emph{essential} if both $A_0$ and $A_1$ are essential. Following Theorem \ref{thm:emptiness}, Theorem \ref{thm:extensible} addresses that the extensibility problem (Problem \ref{prob:extend-problem}) is determined by essential vertex tree-shifts.

\begin{theorem}\label{thm:extensible}
Every local pattern of an essential vertex tree-shift is extensible.
\end{theorem}
\begin{proof}
The desired result follows immediately from the definition of essential vertex tree-shift and analogous discussion to the proof of Theorem \ref{thm:emptiness}, thus it is omitted.
\end{proof}

Theorems \ref{thm:emptiness} and \ref{thm:extensible} reveal that, similar to the well-known results in one-dimensional shift spaces, essential graphs (matrices) declaim the nonemptiness and extensibility of TSFTs. In \cite{BC-2015}, it is seen that each irreducible TSFT has dense periodic points; such a result also holds for one-dimensional SFT. It comes to our mind whether every TSFT contains a periodic point since every one-dimensional SFT has at least one periodic point while there exists an aperiodic two-dimensional SFT (cf.~\cite{CulikII-DM1996, Kari-DM1996}). Theorem \ref{thm:aperiodic-TSFT} infers that there exists an aperiodic TSFT like multidimensional SFTs do.

\begin{theorem}\label{thm:aperiodic-TSFT}
There is an aperiodic tree-shift of finite type.
\end{theorem}
\begin{proof}
Let $\mathcal{A} = \Sigma = \{0, 1\}$. Define
$$
A_0 = \begin{pmatrix}
0 & 1 \\
1 & 0
\end{pmatrix}
\quad \text{and} \quad
A_1 = \begin{pmatrix}
1 & 0 \\
0 & 1
\end{pmatrix}.
$$
It is seen that $X = \mathsf{X}_{A_0, A_1}$ is an aperiodic tree-shift of finite type. Indeed, if $t \in X$ is periodic, then there exists a CPC $P$ such that $\sigma_x t = t$ for all $x \in P$. The construction of $X$ asserts that $t_{w0} \neq t_{w1}$ for every $w \in \Sigma^*$, which derives a contradiction since $P$ is a CPC infers that there exists $w \in \Sigma^*$ such that $w0, w1 \in P$.
\end{proof}

The existence of aperiodic TSFTs illustrates that the tree-shifts exhibit different dynamics from one-dimensional shift spaces. One of the main differences between one-dimensional and multidimensional shift spaces is mixing condition. While there is only one mixing condition for one-dimensional shift spaces, there are several mixing conditions for multidimensional cases; for example, block gluing and strongly irreducible, and etc. The following definition addresses some mixing conditions for the tree-shifts. We remark that a block gluing tree-shift is called a mixing tree-shift in \cite{BC-2015} with a little modification, which is demonstrated as the sufficient condition for the denseness of periodic points. The reader is referred to \cite{BC-2015, BPS-TAMS2010} for more details.

Let $u$ and $v$ be two patterns and let $P$ be a complete prefix code. We say that $u$ and $v$ are connected through $P$ (or $P$ connects $u$ and $v$) if there exists $t \in X$ such that $t|_{S(u)} = u$ and $t|_{wxS(v)} = v$ for every leaf $w$ of $u$ and $x \in P$; herein, $w \in \Sigma^*$ is called a \emph{leaf} of $u$ if $u \in S(u)$ and $ui \notin S(u)$ for each $i \in \Sigma$. We denote the collection of leaves of a finite pattern $u$ by $\mathcal{L}(u)$.

\begin{definition}
A tree-shift is called
\begin{enumerate}[\bf (a)]
\item \emph{topological mixing} (TM) if for any two finite prefix-closed sets $L_1, L_2$ there exists a complete prefix code $P$ such that for any $t^{(1)}, t^{(2)} \in X$ there is a $t \in X$ satisfies $t|_{L_1} = t^{(1)}|_{L_1}$ and $(\sigma_{L_1 x} t)|_{L_2} = t^{(2)}|_{L_2}$ for each $x \in P$;
\item \emph{block gluing} (BG) if there exists a complete prefix code $P$ that connects any two $n$-blocks;
\item \emph{uniformly block gluing} (UBG) if there exists a complete prefix code $P = \Sigma^k$ that connects any two $n$-blocks;
\item \emph{strongly irreducible} (SI) if there exists a complete prefix code $P$ that connects any two patterns;
\item \emph{uniformly strongly irreducible} (USI) if there exists a complete prefix code $P = \Sigma^k$ that connects any two patterns.
\end{enumerate}
\end{definition}

\begin{example}\label{eg:even-shift-USI}
Let $\mathcal{A} = \Sigma = \{0, 1\}$. Let $\mathcal{F} \subset B_2(\mathcal{A}^{\Sigma^*})$ consist of two-blocks $u$ satisfying $u_0 \neq u_1$ or $u_{\epsilon} = u_0 = u_1 = 1$. The TSFT $X = \mathsf{X}_{\mathcal{F}}$ is called \emph{simplified golden-mean tree-shift}; that is, $X$ is generated by
$$
\{(0, 0, 0), (0, 1, 1), (1, 0, 0)\}.
$$
The \emph{even tree-shift} $Y \subset \mathcal{A}^{\Sigma^*}$ is defined as follows.
\begin{enumerate}[\bf (1)]
\item There are even number of $1$'s between two consecutive $0$'s on any path.
\item Any two paths starting at a same node and ending at nodes labeled by $0$ have the same number of $1$'s modulus $2$.
\end{enumerate}
It is easily seen that $Y$ is not a TSFT. Define $\Phi: B_2(X) \to \mathcal{A}$ as
$$
\Phi(u) = \left\{
\begin{aligned}
& 0, && u=(0, 0, 0);\\
& 1, && \hbox{otherwise.}
\end{aligned}
\right.
$$
A careful examination demonstrates that $Y = \phi(X)$, where $\phi = \Phi_{\infty}$; this derives that $Y$ is a sofic tree-shift. Furthermore, it can be verified that $X$ is USI with the complete prefix code $\Sigma^2$ and $Y$ is USI with the complete prefix code $\Sigma^4$. The examination of $X$ being USI is straightforward, thus it is omitted.

Given two patterns $u$ and $v$ which are accepted by $Y$, for each $\ell \in \mathcal{L}(u)$, we divide the discussion into several cases.

\noindent \textbf{Case 1.} $u_{\ell} = 0$ and $v_{\epsilon} = 0$. In this case, $u_{\ell}$ and $v_{\epsilon}$ can be connected through the four-block $\alpha$ satisfying $\alpha_w = 0$ for $w \in \Sigma_4$. More precisely, the pattern $\mu$ satisfying $S(\mu) = S(u) \bigcup \ell \Sigma_4 \bigcup \ell \Sigma^4 S(v)$ and $\mu|_{S(u)} = u, \mu|_{\ell \Sigma_4} = w$, and $\mu|_{\ell \Sigma^4 S(v)} = v$ is accepted by $Y$.\footnote{We remark that, for the simplicity of discussion and notation, $\alpha_{\epsilon}$ coincides with $u_{\ell}$ while $\alpha_w$ is concatenated with $v$ for each $w \in \Sigma^4$.}

\noindent \textbf{Case 2.} $u_{\ell} = 0$, $v_{\epsilon} = 1$. There are two subcases.

\noindent \textbf{Subcase 2-1.} $\min \{|w|: v_w = 0\}$ is even; in other words, there are even number of $1$'s before the first node in $S(v)$ labeled $0$. Let $\alpha \in B_4(Y)$ be the same as discussed in Case 1. It follows that $u_{\ell}$ and $v$ are connected through $\alpha$.

\noindent \textbf{Subcase 2-2.} $\min \{|w|: v_w = 0\}$ is odd. Pick a four-block $\alpha$ which satisfies $\alpha_w = 1$ if and only if $|w| = 3$. Then $u_{\ell}$ and $v$ are connected through $\alpha$.

\noindent \textbf{Case 3.} $u_{\ell} = 1$, $v_{\epsilon} = 0$. Case 3 is divided into two subcases.

\noindent \textbf{Subcase 3-1.} $\ell - \max \{|w|: w \prec \ell, u_w = 0\}$ is even; in other words, there are even number of $1$'s labeled in a path terminated at $\ell$. The selection of $\alpha$ is the same as the one constructed in Case 1.

\noindent \textbf{Subcase 3-2.} $u_{\ell} = 1$, $v_{\epsilon} = 0$, and $\ell - \max \{|w|: w \prec \ell, u_w = 0\}$ is odd. Let $\alpha$ be the four-block satisfying $\alpha_w = 1$ for each $w \in \{\epsilon, 0, 1\}$. Then $\alpha$ connects $u_{\ell}$ and $v$.

\noindent \textbf{Case 4.} $u_{\ell} = 1$, $v_{\epsilon} = 1$. Similar to the discussion above, there are four subcases in Case 4.

\noindent \textbf{Subcase 4-1.} Both $\min \{|w|: v_w = 0\}$ and $\ell - \max \{|w|: w \prec \ell, u_w = 0\}$ are even. The four-block $\alpha$ is the same as considered in Case 1 except $\alpha_{\epsilon} = 1$.

\noindent \textbf{Subcase 4-2.} $\min \{|w|: v_w = 0\}$ is even and $\ell - \max \{|w|: w \prec \ell, u_w = 0\}$ is odd. The four-block $\alpha$ is the same as considered in Case 3-2.

\noindent \textbf{Subcase 4-3.} $\min \{|w|: v_w = 0\}$ is odd and $\ell - \max \{|w|: w \prec \ell, u_w = 0\}$ is even. The four-block $\alpha$ is the same as considered in Case 2-2 except $\alpha_{\epsilon} = 1$.

\noindent \textbf{Subcase 4-4.} Both $\min \{|w|: v_w = 0\}$ and $\ell - \max \{|w|: w \prec \ell, u_w = 0\}$ are odd. Pick a four-block $\alpha$ which satisfies $\alpha_w = 0$ if and only if $|w| = 2$. Then $u_{\ell}$ and $v$ are connected through $\alpha$.

The investigation above asserts that $Y$ is USI with complete prefix code $P = \Sigma^4$.
\end{example}

\begin{proposition}\label{prop:mixing-relation-diagram}
The following diagram holds for any tree-shift.
$$
\begin{array}{ccccc}
USI & \Longrightarrow & SI &  &  \\
\Downarrow &  & \Downarrow &  &  \\
UBG & \Longrightarrow & BG & \Longrightarrow & TM
\end{array}
$$
\end{proposition}

Proposition \ref{prop:mixing-relation-diagram} follows immediately from the definitions of the above five types of mixing properties. It is of interest that if any of the implications can be reversed. Furthermore, it is of more interest that if any inference of the back-diagonal parts exists. Theorem \ref{thm:mixing-relation-TSFT} reveals that $X$ being BG (respectively UBG) is equivalent to $X$ being SI (respectively USI) whenever $X$ is a TSFT.

\begin{theorem}\label{thm:mixing-relation-TSFT}
Suppose $X$ is a tree-shift of finite type. Then $X$ is uniformly strongly irreducible if and only if $X$ is uniformly block gluing, and $X$ is strongly irreducible if and only if $X$ is block gluing.
\end{theorem}
\begin{proof}
It suffices to show that $X$ being UBG infers that $X$ is USI; the inference of BG implying SI can be obtained similarly. Without loss of generality, we may assume that $X$ is a one-step TSFT.

Suppose that $X$ is UBG. Let $k$ be a natural number such that, for any two blocks $u \in B_p(X)$ and $v \in B_q(X)$, there exists $t \in X$ with $t|_{S(u)} = u$ and $t|_{wS(v)} = v$ for each $w \in \Sigma^{k+p-1}$. For any two patterns $\omega$ and $\omega'$ which are accepted by $X$, we extend $\omega'$ to a block and still refer the block to $\omega'$ for the simplicity. For each $\ell \in \mathcal{L}(\omega)$, $X$ being UBG infers that there exists a block $\alpha^{(\ell)}$ accepted by $X$ satisfying $\alpha^{(\ell)}_{\epsilon} = \omega_{\ell}$ and $\alpha^{(\ell)}_w = \omega'$ for each $w \in \Sigma^k$. Since $X$ is a one-step TSFT, there exists a pattern $\varpi$ accepted by $X$ satisfying that $\varpi_{S(\omega)} = \omega$ and $\varpi_{\ell w S(\omega')} = \omega'$ for each $\ell \in \mathcal{L}(\omega), w \in \Sigma^k$. Namely, $X$ is USI. The proof is complete.
\end{proof}

It is known that each irreducible one-dimensional SFT exhibits positive topological entropy. For multidimensional shift spaces, Boyle \emph{et al.}~\cite{BPS-TAMS2010} demonstrates that each block gluing SFT has positive topological entropy. Theorem \ref{thm:BG-positive-entropy} reveals that the block gluing condition is sufficient for a TSFT being of positive topological entropy.

\begin{theorem}\label{thm:BG-positive-entropy}
Suppose that a tree-shift of finite type is block gluing. Then it exhibits positive topological entropy.
\end{theorem}
\begin{proof}
Suppose that $X$ is a block gluing TSFT over $\mathcal{A}$ with $|\mathcal{A}| = \kappa$. Let $P$ be a CPC which connects any two blocks with $|P| = k$ for some $k \in \mathbb{N}$; recall that $|P| = \max \{|x|: x \in P\}$. It is seen that $|B_{k+1} (X)| \geq \kappa^2$. Indeed, for $u, v \in \mathcal{A}$, there exists $t \in X$ such that $t_{\epsilon} = u$ and $t_x = v$ for every $x \in P$. Since $X$ is shift invariant and is a TSFT, it follows that
$$
|B_{2 k + 1} (X)| \geq \kappa^2 \cdot \kappa^{4} = \kappa^6;
$$
inductively, we derive that
$$
B_{\ell k + 1} (X)| \geq \kappa^2 \cdot \kappa^{2^2} \cdot \cdots \cdot \kappa^{2^{\ell}} = \kappa^{2 (2^{\ell} - 1)}.
$$
Therefore,
$$
h(X) \geq \lim_{\ell \to \infty} \dfrac{\ln^2 \kappa^{2 (2^{\ell} - 1)}}{\ell k + 1} = \ln 2 > 0.
$$
This completes the proof.
\end{proof}

\begin{remark}
In the proof of Theorem \ref{thm:BG-positive-entropy}, it is seen that block gluing TSFTs not only exhibit positive but full topological entropy. Such a phenomenon indicates that a softer condition might be sufficient for yielding positive topological entropy. On the other hand, irreducibility is not sufficient for attending positive topological entropy. For example, let
$$
A_0 = A_1 = \begin{pmatrix}
0 & 1 \\
1 & 0
\end{pmatrix}.
$$
It comes immediately that $\mathsf{X}_{A_0, A_1}$ is an irreducible TSFT and $h(\mathsf{X}_{A_0, A_1}) = 0$.
\end{remark}

\begin{proposition}\label{prop:SI-not-UBG-not-SI}
A strongly irreducible tree-shift $X$ may not be uniformly block gluing even when $X$ is a tree-shift of finite type. On the other hand, a uniformly block gluing tree-shift may not be strongly irreducible in general.
\end{proposition}

Example \ref{eg:SI-not-UBG} illustrates a TSFT which is strongly irreducible but not uniformly block gluing. Combining with Theorem \ref{thm:mixing-relation-TSFT} provides an example, which distinguishes block gluing from uniformly block gluing. On the other hand, Example \ref{eg:UBG-not-SI} yields a tree-shift which is uniformly block gluing but not strongly irreducible. Such a novel phenomenon is not observed in multidimensional shift spaces.

\begin{example}[SI does not imply UBG]\label{eg:SI-not-UBG}
Let $\mathcal{A} = \Sigma = \{0, 1\}$ and let $\mathcal{F}$ consist of those two blocks $u$ with $u_{\epsilon} = u_0 = u_1$; more explicitly, $\mathcal{F} = \{(0, 0, 0), (1, 1, 1)\}$. It is seen that the tree-shift of finite type $\mathsf{X}_{\mathcal{F}}$ is strongly irreducible with complete prefix code $P = \{0, 10, 11\}$. To see that $\mathsf{X}_{\mathcal{F}}$ is not uniformly block gluing, we consider the two-block $u = (0, 1, 0) \in B_2(\mathsf{X}_{\mathcal{F}})$. Suppose that $\mathsf{X}_{\mathcal{F}}$ is uniformly block bluing with a complete prefix set $\Sigma^k$ for some $k \in \mathbb{N}$. Then there is a block $v \in B(\mathsf{X}_{\mathcal{F}})$ satisfies $v|_{\Sigma_1} = u$ and $v|_{w \Sigma_1} = u$ for each $w \in \Sigma^{k+1}$. In other words, $v_w = 0$ for each $w \in \Sigma^{k+1}$; this makes $v_w = 1$ for each $w \in \Sigma^k$. Repeating the process infers that
$$
v_w = \left\{
\begin{aligned}
&0, && |w| = k+1-2i, i \geq 0; \\
&1, && |w| = k - 2i, i \geq 0.
\end{aligned}
\right.
$$
Namely, $v_w = v_{w'}$ if $|w| = |w'| \leq k+1$, which contradicts to the fact that $v_0 = 1$ and $v_1 = 0$. Hence $\mathsf{X}_{\mathcal{F}}$ is not uniformly block gluing.
\end{example}

\begin{example}[UBG does not imply SI]\label{eg:UBG-not-SI}
To construct a tree-shift $X$ which is uniformly block gluing but not strongly irreducible, Theorem \ref{thm:mixing-relation-TSFT} asserts that $X$ can not be a TSFT. Let $\mathcal{A} = \Sigma = \{0, 1\}$ and let
$$
\mathcal{F}_n = \{u: S(u) = \Sigma_n, u_x \neq u_y \text{ for some } x \neq y \in \Sigma_n\}
$$
for $n \geq 0$, and let $\mathcal{F} = \bigcup\limits_{n \geq 0} \mathcal{F}_n$. For each $t \in \mathsf{X}_{\mathcal{F}}$, it comes that $t_x = t_y$ if $|x| = |y|$. It can be verified without difficulty that $\mathsf{X}_{\mathcal{F}}$ is uniformly block gluing with complete prefix code $\{0, 1\}$. However, $\mathsf{X}_{\mathcal{F}}$ is not strongly irreducible. Indeed, suppose that $\mathsf{X}_{\mathcal{F}}$ is strongly irreducible with a complete prefix code $P$. Let $w \in P$. Consider a pattern $u$ whose support is the collection of all subwords of $\overline{w}$, where $\overline{a} = 1 - a$ for $a \in \mathcal{A}$, and $u_x = 0$ for each $x \in S(u)$. Pick $v = 1$; then there is a $(k+1)$-block $\mu \in \mathsf{X}_{\mathcal{F}}$ satisfies $\mu_w = 1$ and $\mu_{\overline{w}} = 0$, which is forbidden for $\mathsf{X}_{\mathcal{F}}$. Therefore, $\mathsf{X}_{\mathcal{F}}$ is not strongly irreducible.
\end{example}

An immediate inference of Propositions \ref{prop:mixing-relation-diagram} and \ref{prop:SI-not-UBG-not-SI} is that, generically, a BG (respectively UBG) tree-shift may not be SI (respectively USI); this is illustrated in Corollary \ref{cor:BG-not-SI-UBG-not-USI}.

\begin{corollary}\label{cor:BG-not-SI-UBG-not-USI}
A block gluing (respectively uniformly block gluing) tree-shift may not be strongly irreducible (respectively uniformly strongly irreducible) in general.
\end{corollary}

\section{Conclusions}\label{sec:conclusion}

This paper investigates some fundamental properties of tree-shifts such as the emptiness problem, the extensibility problem, the existence of periodic points, and the sufficient condition of exhibiting positive topological entropy. It turns out the tree-shifts stand alone from one-dimensional and multidimensional shift spaces; this makes the tree-shifts an appropriate approach for elucidating multidimensional shift spaces. Table \ref{table:summary} summarizes the comparison of tree-shifts of finite type, one-dimensional and multidimensional shifts of finite type.

\begin{table}
\begin{center}
\begin{tabular}{c|ccc}
\rule[-1.5ex]{0pt}{2.5ex}  & $1$-d SFT & TSFT & $k$-d SFT \\ 
\hline
\rule[-1ex]{0pt}{4ex} (1) & Decidable & Decidable & Undecidable \\ 
\rule[-1ex]{0pt}{2.5ex} (2) & True & True & False \\ 
\rule[-1ex]{0pt}{2.5ex} (3) & True & False & False \\ 
\rule[-1ex]{0pt}{2.5ex} (4) & Irreducible & Irreducible & Block gluing \\ 
\rule[-1ex]{0pt}{2.5ex} (5) & Perron number & Perron number & Right-recursively-enumerable \\ 
\rule[-1ex]{0pt}{2.5ex} (6) & Irreducible & Block gluing & Block gluing \\ 
\end{tabular} 
\end{center}
\caption{Comparison of TSFTs, one-dimensional and multidimensional SFTs. \newline (1) Emptiness problem: For a given forbidden set $\mathcal{F}$, does there exist an algorithm determining $\mathsf{X}_{\mathcal{F}} \neq \varnothing$?
\newline (2) Extensibility problem: Does every local pattern extend to a global pattern?
\newline (3) Existence of periodic points: Does $\mathsf{X}_{\mathcal{F}}$ contain periodic points provided $\mathsf{X}_{\mathcal{F}} \neq \varnothing$?
\newline (4) Denseness of periodic points: Under what condition are the periodic points dense?
\newline (5) Topological entropy: What kind of algebraic properties does the topological entropy satisfy?
\newline (6) Positive topological entropy: Under what condition does $\mathsf{X}_{\mathcal{F}}$ exhibit positive topological entropy?}
\label{table:summary}
\end{table}

Remarkably, the denseness problem of periodic points remains to be open for $k$-dimensional SFTs when $k \geq 3$ even for strongly irreducible SFTs; meanwhile, periodic points are dense in block gluing tree-shifts and irreducible TSFTs (cf.~\cite{BPS-TAMS2010, BC-2015}). Furthermore, the topological entropy of each TSFT is a Perron number (\cite{ABC-2016}), and the topological entropy of a multidimensional SFT is known as right recursively enumerable (\cite{HM-AoM2010}).

Except for the above, this paper also investigates several conditions of mixing, say, topological mixing, block gluing, uniformly block gluing, strongly irreducible, and uniformly strongly irreducible. It comes that a TSFT is strongly irreducible (respectively uniformly strongly irreducible) if and only if it is block gluing (respectively uniformly block gluing). This is another difference between tree-shifts and multidimensional shift spaces. Generally, neither uniformly block gluing nor strongly irreducible tree-shifts imply one another. To sum up, the relations between these mixing conditions are revealed as follows.
$$
\begin{array}{ccccc}
USI & \Rightarrow & SI &  &  \\
\Downarrow & \rotatebox{45}{$\nLeftrightarrow$} & \Downarrow &  &  \\
UBG & \Rightarrow & BG & \Rightarrow & TM
\end{array}
$$
Nevertheless, the diagram reduces to
$$
UBG (=USI) \Rightarrow BG (=SI) \Rightarrow TM
$$
whenever we restrict the discussion to TSFTs.

One of the main results of this paper is proving that, for a TSFT, the block gluing condition is the sufficient condition for exhibiting positive topological entropy. Such a condition coincides with the one required in multidimensional SFTs; in the mean time, uniform filling property, instead of block gluing condition, asserts the entropy minimality (cf.~\cite{BPS-TAMS2010}). It is of interest that, for TSFTs, if blocking gluing condition implies entropy minimality. The related work is under preparation.

\bibliographystyle{amsplain}
\bibliography{../../grece}

\end{document}